\documentclass[a4paper,12pt]{article}
\usepackage{amssymb}
\usepackage{amsmath}
\usepackage{amsfonts}
\usepackage{amsthm}
\usepackage{graphicx,epsfig}
\usepackage{url}
\usepackage{enumerate}
\usepackage{dsfont}
\usepackage{pst-all}
\usepackage{bbm} 
\usepackage{setspace}
\usepackage{xcolor,rotating,epic,eepic}
\usepackage[T1]{fontenc} 
\usepackage{dsfont}\usepackage[left=2cm, right=2cm, top=3cm, bottom=3.5cm]{geometry}
  
\newtheorem{lemma}{Lemma}      
   
\newtheorem{proposition}[lemma]{Proposition}

\newtheorem{remark}{Remark}

\DeclareMathOperator{\length}{length}

\title {Statistical properties of intermittent maps with unbounded derivative}

\begin{document}
\author {Giampaolo \textsc{Cristadoro},~Nicolai \textsc{Haydn},~ Philippe \textsc{Marie},~ Sandro \textsc{Vaienti} }
\maketitle

\bibliographystyle{plain}

\abstract{ We study the ergodic and statistical properties of a class of maps of the circle and of the interval of Lorenz type which present indifferent fixed points and points with unbounded derivative. These maps have been previously investigated in the physics literature. We prove in particular that correlations decay 
polynomially, and that suitable Limit Theorems (convergence to Stable Laws or Central Limit Theorem) 
hold for H\"older continuous observables.  We moreover
show that the return and hitting times are in the limit exponentially distributed.}
$$$$
$$$$

\section{Introduction}
The prototype for intermittent maps of the interval is the well known Pomeau-Manneville map $T$ defined on the unit interval $[0,1]$ and which admits a neutral fixed point at $0$ with local behavior $T(x)=x+cx^{1+\alpha}$; otherwise it is uniformly expanding. The constant $\alpha$ belongs to $(0,1)$ to guarantee the existence of a finite absolutely continuous invariant probability measure  and the constant $c$ could be chosen in such a way that the map $T$ has a Markov structure. This map enjoy polynomial decay of correlations  and this property still persists even if  the map is not anymore Markov \cite{LSY}. 

Another interesting class of maps of the interval are the one-dimensional uniformly expanding Lorenz-like maps (see \cite{GW, W, GS} for their introduction and for the study of their topological properties),  whose features are now the presence of points with unbounded derivatives and the lack of Markov structure: in this case one could build up towers and find various rates for the decay of correlations depending on the tail of the return time function on the base of the tower, see, for instance \cite{KDO} and \cite{HLO}. The latter paper deals in particular with one-dimensional maps which admit critical points and, eventually, points with unbounded derivatives, but it leaves open the case where there is  presence of neutral fixed points.

In this paper we are interested  in maps which exhibit the last two  behaviors, namely neutral fixed points and points with unbounded derivatives. Such maps have been introduced into the physics literature by Grossmann and Horner in 1985 \cite{GH}; they showed numerically a polynomial decay of correlations and they also studied other statistical properties like the susceptibility and the $1/f$-noise. Another contribution by A. Pikovsky \cite{P} showed, still with heuristic arguments, that these maps produce anomalous diffusion with square displacement growing faster than linearly.  R. Artuso and G. Cristadoro \cite{AC} improved the latter result by computing the moments of the displacement  on the infinite replicas of the fundamental domain and showed a phase transition in the exponent of the moments growth.  Recently  Lorenz cusp maps arose to describe the distribution of the Casimir maximum in the Kolmogorov-Lorenz model of geofluid dynamics \cite{PM}. Despite this interesting physical phenomenology, we did not find any rigorous mathematical investigation of such maps. These maps are defined on the torus $\mathbb{T}=[-1,1]/\sim$ and depend on the parameter $\gamma$ (see below); when $\gamma=2$ the corresponding map was taken as an example of the non-summability of the first hyperbolic time by Alves and Araujo in \cite{AA}. This maps reads:
\begin{equation}\label{AA}
\tilde{T}(x)=\left\{
\begin{aligned}
& 2\sqrt{x}-1  &\hspace{0.6cm} \text{if} \ x\ge 0\\
&  1-2\sqrt{|x|} &\hspace{0.6cm} \text{otherwise}   
\end{aligned}
\right.
\end{equation}
 and it was  proved in \cite{AA} that it is topological mixing, but no other ergodic properties were studied.

Actually, the Grossmann and Horner maps are slightly different from those investigated in \cite{P} and \cite{AC}, the difference being substantially in the fact that the latter are defined on the circle instead than on the unit interval. We will study in detail the circle version of these maps in Sections 2 to 5, and we will show in Section 6 how to generalize our results to the interval version: since both classes of maps are Markov, the most important information, especially in computing distortion, will come from the local behavior around the neutral fixed points and the points with unbounded derivatives and these behaviors will be the same for both versions. There is nevertheless an interesting difference. The circle version introduced in Section 1 is written in such a way that the Lebesgue measure is automatically invariant. This is not the case in general for the interval version quoted in Section 6. However the strategy that we adopt to prove statistical properties (Lai-Sang Young towers) will give us as well the existence of an absolutely continuous invariant measure and we will complete it by providing  informations on the behavior of the density. It is interesting to observe that in the class of maps considered by Grossmann and Horner on the interval $[-1,1]$ (see Sect. 6), the analog of (\ref{AA})  is given by the following map:

\begin{equation}\label{HH}
\tilde{S}(x)= 1-2\sqrt{|x|} \, .
\end{equation}
This map was investigated by Hemmer in 1984 \cite{HH}: he also computed by  inspection the invariant density which is $\rho(x)=\frac{1}{2}(1-x)$ and the Lyapunov exponent (simply equal to $1/2$), but he only argued about a slow decay of correlations. We will show in Sect. 6 how to recover the qualitative behavior of this density (and of all the others in the Grossmann and Horner class).\\

In this paper we study the one-parametric family of continuous maps $T$ (Fig.~1) which are $C^1$ on $\mathbb{T}/\{0\}$, $C^2$ on $\mathbb{T}/(\{0\}\cup \{1\})$
 and  are implicitly defined  on the circle by the equations:
$$
x=\left\{
\begin{aligned}
& \frac{1}{2\gamma}(1+T(x))^{\gamma}   & \hspace{0.6cm} \text{if}  \hspace{0.6cm} 0\le x\le\frac{1}{2\gamma}\\
& T(x)+\frac{1}{2\gamma}(1-T(x))^{\gamma} & \hspace{0.6cm} \text{if}   \hspace{0.6cm}\frac{1}{2\gamma}\le x\le 1
\end{aligned}
\right.
$$
and for negative values of $x$ by putting $T(-x)=-T(x)$. 
We assume that parameter $\gamma>1$. Note that when $\gamma=1$ the map is continuous  with constant derivative equal to $2$ and is the classical doubling map.  The point 1 is a fixed point with derivative equal to $1$, while at 0 the derivative becomes infinite. The map leaves the Lebesgue measure $m$  invariant (it is straightforward to check that the Perron-Frobenius operator has $1$ as a fixed point). We will prove in the next sections the usual bunch of statistical properties: decay of correlations (which, due to the parabolic fixed point, turns out to be polynomial with the rate found in~\cite{GH}); convergence to Stable Laws   
and  large deviations; statistics of recurrence. All these results will follow from existing techniques, especially  towers, combined  with the distortion bound  proved in the next section. Distortion will in fact allows us to  induce with the first return map on each cylinder of a countable Markov partition associated to $T$.
Actually one could induce on a suitable interval only (called $I_0$ in the following):
the proof we give is intented to provide disortion on {\em all} cylinders of the countable Markov partition covering {\em mod $0$} the whole space $[-1,1]$, since this is necessary in order to apply the inducing technique of \cite{BSTV} which will give us the statistical features of recurrence studied in Sect. 5: distributions of first return and hitting times, Poissonian statistics for the number of visits, extreme values laws.

\section{Distortion}

{\emph Notations}: With $a_n\approx b_n$ we mean that there exists a constant $C\ge 1$ such that $C^{-1}b_n\le a_n\le C b_n$ for all $n\ge 1$;  with $a_n\lesssim b_n$ we mean that there exists a constant $C\ge 1$ such that $\forall n\ge 1$, $a_n\le C b_n$; with $a_n\sim b_n$ we mean that $\lim_{n\to \infty} \frac{a_n}{b_n} =1$.
 We will also use the symbol "${\cal O}$" in the usual sense. Finally we denote with $|A|$ the diameter of the set $A$.\\

 There is a countable Markov partition $\{I_m\}_{m\in\mathbb{Z}}$  associated to this map; the partition is built $mod$  $m$  as follows:
$I_m=(a_{m-1},a_m)$ for all $m\in\mathbb{Z}^{\ast}$ and $I_0=(a_{0-},a_{0+})/\{0\}$, 
 where, denoting with $T_+=T_{|(0,1)}$ and with $T_-=T_{|(-1,0)}$:
\begin{equation}
a_{0+}=\frac{1}{2\gamma}~~,~~ a_{0-}=-\frac{1}{2\gamma} ~~~~\text{and}~~~~ a_i=T_+^{-i}a_{0+}~~,~~ a_{-i}=T_-^{-i}a_{0-},\quad   i\ge 1 \, .\nonumber 
 \end{equation}
\newline
Then we define $\forall i\ge 1$:
\begin{equation}
b_{-i}=T_-^{-1}a_{i-1}~~~~~~\text{and}~~~~b_{i}=T_+^{-1}a_{-(i-1)} \, .\nonumber
\end{equation}
\begin{figure}[t]
\centerline{\epsfxsize=12.cm \epsfbox{piko-bn.eps}}
\end{figure}
\newline

We  now state without proof a few results which are direct consequences of the definition of the map.
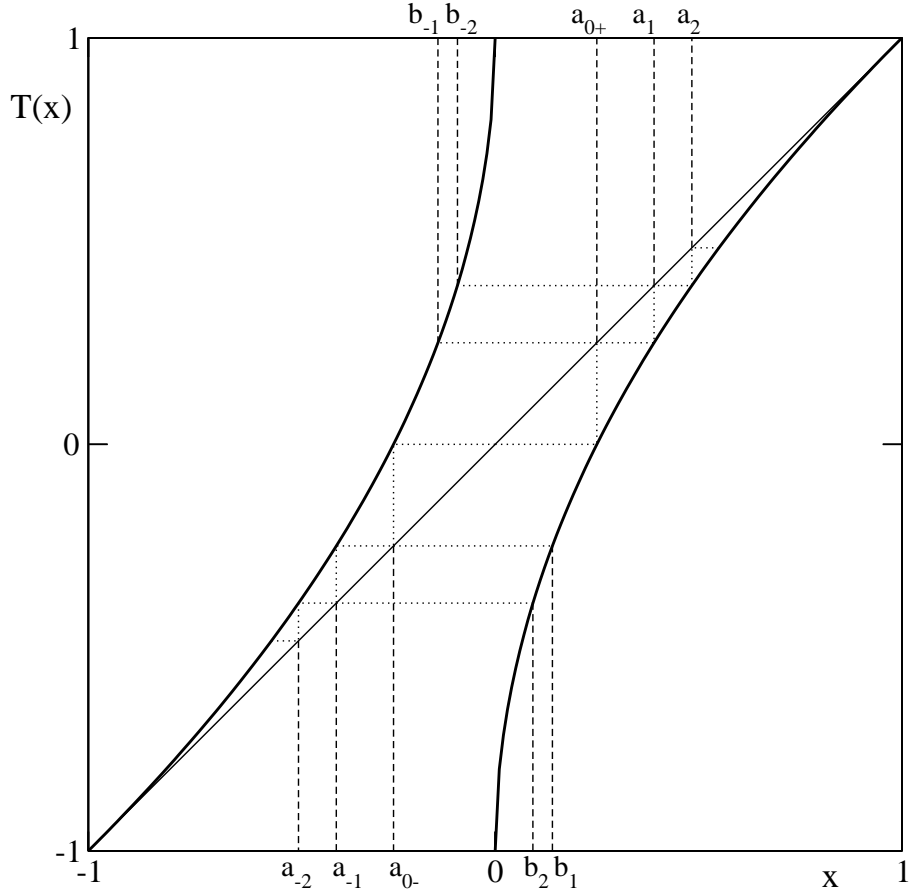
\begin{lemma}\label{L1}
\begin{enumerate}
\item When $x\rightarrow 1^-$: $T(x)= 1-(1-x)-\frac{1}{2\gamma}(1-x)^{\gamma}+$${\cal O}$ $((1-x)^{\gamma})$
\item When $x\rightarrow 0^+$: $T(x)= -1+(2\gamma)^{\frac{1}{\gamma}}x^{\frac{1}{\gamma}}\, .$ 
\end{enumerate}
\end{lemma}

\begin{lemma}\label{L2} 
We have for all $n\ge0$, $a_{\pm(n+1)}=a_{\pm n}+\frac{1}{2\gamma}(1-a_{\pm n})^{\gamma}$ and:
\begin{eqnarray}
a_n&\sim& ~1-\left(\frac{2\gamma}{\gamma-1}\right)^{\frac{1}{\gamma-1}}\frac{1}{n^{\frac{1}{\gamma-1}}}\nonumber\\
a_{-n}&\sim& -1+\left(\frac{2\gamma}{\gamma-1}\right)^{\frac{1}{\gamma-1}}\frac{1}{n^{\frac{1}{\gamma-1}}}\nonumber\\
l_n:=\length[a_{n-1},a_{n}]&\sim& \frac{1}{2\gamma}\left(\frac{2\gamma}{\gamma-1}\right) ^{\frac{\gamma}{\gamma-1}}\frac{1}{n^{\frac{\gamma}{\gamma-1}}} \quad \quad{n>1}\nonumber\\
\vert b_{\pm (n+1)}\vert&\sim& \frac{1}{2\gamma}\left(\frac{2\gamma}{\gamma-1}\right) ^{\frac{\gamma}{\gamma-1}}\frac{1}{n^{\frac{\gamma}{\gamma-1}}}, \quad \quad{n>1} \, .\nonumber
\end{eqnarray}
\end{lemma}
We now induce on the interval $I_m:=(a_{-m},a_m)/\{0\}$  and provide a  bounded distortion estimate for the first return map.
We define $Z_{m,p}=Z^+_{m,p}\cup Z^-_{m,p}$, where: $Z^{+}_{m,1}:=(b_{m+1},a_{m})$, $Z^{-}_{m,1}:=(a_{-m}, b_{-(m+1)})$ and $Z^{+}_{m,p>1}:= (b_{m+p},b_{m+p-1})$, $Z^{-}_{m,p>1}:= (b_{-(m+p-1)},b_{-(m+p)})$. 
Note that $I_m=\cup_{p\ge1}Z_{m,p}$ and that the  first return map   $\widehat{T}=I_m\rightarrow I_m$  acts on each  $Z_{m,p}$ as $\widehat{T}=T^p$ and in particular:
\begin{equation*}
T^p(Z^+_{m,p})=\left\{ 
\begin{aligned}
&  (a_{-m},a_{m-1})\quad &  p=1 \\ 
& (a_{-m},a_{-(m-1)})\quad & p>1
\end{aligned}
\right.
\qquad \quad
T^p(Z^-_{m,p})=\left\{ 
\begin{aligned}
&  (a_{-(m-1)},a_{m})\quad & p=1&  \\ 
& (a_{m-1},a_{m})\quad   & p>1& \, .
\end{aligned}
\right.
\end{equation*}
We finally  observe that the induced map $\widehat{T}$  is uniformly expanding in the sense that for each $m$ and $p$  there exists $\beta>1$ such that $|D\widehat{T}(x) |>\beta$, $\forall x\in I_m$.\footnote{Using the chain rule we can see  that  $\beta \equiv \inf_{x \in Z_{m,1}} |DT(x)|>1$.} 
\begin{proposition}[Bounded distortion]
Let us induce on $I_m$; then there exists a constant $K>0$ that depends on m, such that for each $m$ and $p$ and for all $x,y\in Z_{m,p}$, we have:
$$
\left|\frac{DT^p(x)}{DT^p(y)}\right|\leq e^{K |T^p(x)-T^p(y)| }\le e^{2K}\, .
$$
\end{proposition}
\begin{remark} \label{remark}
The cylinder $Z_{m,p}$ is the disjoint union of the two open intervals $Z^+_{m,p}$ and $ Z^-_{m,p}$ sitting on the opposite sides of $0$ (see above). Whenever $x$ and $y$ belongs to  different components, we proceed by first noticing that $DT^p(x)=DT^p(-x)$ and $-x$ sits now in the same component as $y$. By exploiting the fact that $T^p$ is odd we get
$$
\left|\frac{DT^p(x)}{DT^p(y)}\right|=\left|\frac{DT^p(-x)}{DT^p(y)}\right|\leq e^{ K |T^p(-x)-T^p(y)| }\leq e^{ K|-T^p(x)-T^p(y)| } \leq e^{ K |T^p(x)-T^p(y)| } \, .
$$
and we can thus concentrate on the case when $x$ and $y$ are taken in the same open component (see below).
\end{remark}
\begin{proof}  We denote with $l_m$ the length of the interval $(a_{m-1}, a_m)$ (when $m=0$, $l_0$= length of $ (0,a_{0+})$). 
We start by observing that
\begin{eqnarray}
\left|\frac{DT^{p}(x)}{DT^{p}(y)}\right|&=&\exp\left[ \sum_{q=0}^{p-1}\left(\log |DT(T^qx)|-\log |DT(T^qy)| \right) \right]\nonumber\\
&=&\exp\left[ \sum_{q=0}^{p-1}\left|\frac{D^2T(\xi)}{DT(\xi)}\right||T^qx-T^qy|\right] \, ,\label{B0}
\end{eqnarray}
where $\xi$ is a point between $T^qx$ and $T^qy$.

We divide the cases $p=1$ and $p>1$.
\begin{itemize}
\item $p=1$
\newline
For $(x,y)\in Z^-_{m,1}$ (see Remark(\ref{remark}) above),  using $|x-y|<|T(x)-T(y)|$, we directly get:\\
$$
\left|\frac{DT(x)}{DT(y)}\right| \leq \exp{\left[ K_1|T(x)-T(y)| \right]}\, ,
$$
where  $K_1=\sup_{(Z^-_{m,1})} D^2T=D^2T(a_m)$.
\item $p>1$
\newline
Start with $x,y \in Z^{-}_{m,p}$ (see Remark(\ref{remark}) above), then $Tx,Ty\in(a_{m+p-2},a_{m+p-1})$; $T^2x,T^2y\in (a_{m+p-3},a_{m+p-2})$; $ ~\ldots~$; $T^{p-1}x,T^{p-1}y\in (a_{m},a_{m+1})$, we have:
\begin{eqnarray}
(\ref{B0})&\leq& \exp\left[ \frac{\sup_{(Z^{-}_{m,p})}\big(|D^2T|\big)}{\inf_{(Z^{-}_{m,p})}\big(|DT|\big)}|x-y|+ \sum_{q=1}^{p-1}\frac{\sup_{(a_{m+p-q-1},a_{m+p-q})}\big(|D^2T|\big)}{\inf_{(a_{m+p-q-1},a_{m+p-q})}\big(|DT|\big)}|T^qx-T^qy|\right]\nonumber\\
&\leq&   \exp \left[ \sup_{(Z^{-}_{m,p})}\big( |D^2T| \big) |x-y| + \sum_{q=1}^{p-1}  \sup_{(a_{m+p-q-1},a_{m+p-q})}{\big(|D^2T| \big)} |T^qx-T^qy| \right] \, . \label{B1}
\end{eqnarray}
\newline
To continue we need the following
\begin{lemma}\label{L3}
For $x,y \in Z^-_{m,p}$ (see Remark(\ref{remark}) above) we have:
\newline
\\
(i)  $\quad \sum_{q=1}^{p-1}  \sup_{(a_{m+p-q-1},a_{m+p-q})}\big( |D^2T| \big) |T^qx-T^qy| \, \le \,   C_1 |T^{p-1}Z| $\\
\\
(ii) $\quad \sup_{Z^-_{m,p}} \big(|D^2T| \big) |x-y| \, \le \,  C_2\frac{|T^{p-1}Z|}{ l_{m+1} } \, ,  $\\
\newline
where we set for convenience  $Z$ the interval with endpoints $x$ and $y$.
\end{lemma}
\begin{proof}
(i) Denote $T^{p-1}x=z_x$ and $T^{p-1}y=z_y$; since the derivative is  decreasing on $(0,m)$ we have:
\begin{eqnarray}\label{B2}
 |T^qx-T^qy|\le  \frac{1}{DT^{p-1-q}(a_{m+p-q})}|z_x-z_y| \, .
\end{eqnarray}
Let's now consider the term:
\begin{equation}\label{B3}
DT^{p-1-q}(a_{m+p-q})=DT(a_{m+p-q})DT(Ta_{m+p-q})\ldots DT(T^{p-2-q}a_{m+p-q}) \, .
\end{equation}
Since for $q\ge 1$ and  $\xi_1\in (a_{q},a_{q+1})$:
$$
DT(a_{q})\geq DT(\xi_1)=\frac{T(a_{q+1})-T(a_{q})}{a_{q+1}-a_{q}}=\frac{a_{q}-a_{q-1}}{a_{q+1}-a_{q}} 
$$
it follows that
$$
(\ref{B3})\geq \frac{a_{m+p-q}-a_{m+p-q-1}}{a_{m+p+1-q}-a_{m+p-q}} \cdot \frac{a_{m+p-q-1}-a_{m+p-q-2}}{a_{m+p-q}-a_{m+p-q-1}}\ldots \frac{a_{m+2}-a_{m+1}}{a_{m+3}-a_{m+2}}
\geq \frac{a_{m+2}-a_{m+1}}{a_{m+p+1-q}-a_{m+p-q}} 
$$
and thus:\footnote{We have just proved that if $\xi$ is any point in $(a_{m+p},a_{m+p+1})$ 
(and the same result holds for its negative counterpart  $(a_{-(m+p+1)},a_{-(m+p)})$ as well) then $DT^p(\xi)\ge \frac{a_{m+2}-a_{m+1}}{a_{m+p+1}-a_{m+p}}$. In a similar way we can prove the lower bound: $DT^p(\xi)\le \frac{a_{0^+}}{a_{m+p-1}-a_{m+p-2}}$, for $p\ge 2$.
}
$$
\frac{1}{DT^{p-1-q}(a_{m+p-q})}\leq \frac{a_{m+p+1-q}-a_{m+p-q}}{a_{m+2}-a_{m+1}}.
$$
Moreover: $|z_x-z_y|\leq \vert T^{p-1}Z\vert$.
Finally:
\begin{eqnarray}\label{den}
(\ref{B2})\leq  \: \frac{a_{m+p+1-q}-a_{m+p-q}}{a_{m+2}-a_{m+1}}\:\vert T^{p-1}Z\vert \, .
\end{eqnarray}
Using lemmas \ref{L1} and \ref{L2}  we see that there exists a constant $C_0$ depending only on the map $T$ such that:
$$
\big( \sup_{(a_{m+q-1},a_{m+q})}|D^2T|\big)\: (a_{m+q+1}-a_{m+q})\le C_0\cdot \frac{1}{(q+m)^{\frac{\gamma-2}{\gamma-1}}(q+m)^{\frac{\gamma}{\gamma-1}}}= C_0\cdot\frac{1}{(q+m)^2}.
$$
Therefore the sum over $q=1,2,\dots$ is summable and there exists 
a constant $C_1$ such that for $x,y\in Z^-_{m,p}$:
\begin{equation}\label{d1}
\sum_{q=1}^{p-1} \big( \sup_{(a_{m+p-q-1},a_{m+p-q})} |D^2T| \big) |T^qx-T^qy| \leq  C_1\vert T^{p-1}Z\vert \, .
\end{equation}

(ii) In this case we need to control the behavior of the map close to $0$.  In particular,
by  using lemmas \ref{L1} and \ref{L2} (and the symmetry of $b_ {\pm i}$) we start by noticing  that 
\begin{eqnarray}\label{sd}
\big(\sup_{(b_{i+1},b_{i})}|D^2T|\big)|b_{i}-b_{i+1}|={\cal O}( \frac{i^{\frac{2\gamma-1}{\gamma-1}}}{i^{\frac{2\gamma-1}{\gamma-1}}})= 1.
\end{eqnarray}
Combining (\ref{d1}) and (\ref{sd}) with (\ref{B1}) we get that 
 there exists a constant $D_2$ so that for all $j\leq p-1$
\begin{eqnarray}\label{B5}
\frac{1}{D_2}\leq \left|\frac{DT^j(x)}{DT^j(y)}\right|\leq D_2.
\end{eqnarray}

Let's call $\alpha=b_{-(m+p-1)}, \beta=b_{-(m+p)}$ the end points of $Z^-_{m,p}$.
For $j_1,j_2\leq p-1$ there exist $\eta_1\in (x,y)$ and $\eta_2\in (\alpha,\beta)$ such that:
\begin{align*}
|T^{j_1}x-T^{j_1}y|&=DT^{j_1}(\eta_1)|x-y|,\\
|T^{j_2}\alpha-T^{j_2}\beta|&=DT^{j_2}(\eta_2)|\alpha-\beta|.
\end{align*}
The distortion bound (\ref{B5}) yields
$$
\frac{\left|T^{j_1}x-T^{j_1}y\right|}{\left|T^{j_1}\alpha-T^{j_1}\beta\right|}\leq D_2^2\frac{\left|T^{j_2}x-T^{j_2}y\right|}{\left|T^{j_2}\alpha-T^{j_2}\beta\right|}.
$$
If we now choose $j_1=0$ and $j_2=p-1$ then
$$
\big(\sup_{(\alpha,\beta)}|D^2T|\big) |x-y| \le D_2^2\big(\sup_{(\alpha,\beta)}|D^2T|\big)\frac{|\alpha-\beta|\cdot |T^{p-1}x-T^{p-1}y|}{|T^{p-1}\alpha-T^{p-1}\beta|} \, .
$$
Since $|T^{p-1}\alpha-T^{p-1}\beta|=l_{m+1}=a_{m}-a_{m+1}$  and  $x$ and $y$ to belong to $Z$ we get:
$$
\big(\sup_{(\alpha,\beta)}|D^2T|\big) |x-y| \le D_2^2\big(\sup_{(\alpha,\beta)}|D^2T|\big)\frac{|\alpha-\beta|\cdot |T^{p-1}Z|}{ l_{m+1} } 
$$
and using distortion bound (\ref{sd}) once more we have that there exist a constant $C_2$ such that:
$$
\big(\sup_{(\alpha,\beta)}|D^2T|\big) |x-y| \le C_2 \frac{|T^{p-1}Z|}{ l_{m+1} } \, .
$$

\end{proof}
\end{itemize}
By collecting lemma \ref{L3}(i) and \ref{L3}(ii) we see that the ratio $|DT^p(x)/DT^p(y)|$, $(x,y \in Z)$ is bounded as:
\begin{eqnarray}\label{d}
\left|\frac{DT^p(x)}{DT^p(y)}\right|\le \exp\left[C_2\frac{\vert T^{p-1}Z \vert}{l_{m+1}}+C_1\vert T^{p-1}Z \vert \right]\le \exp\left[K_2 \vert T^{p-1}Z \vert \right]
\end{eqnarray}
with $K_2=C_1+C_2/l_{m+1}$.

We finish the proof of the Proposition by choosing $K=\max (K_1 ,K_2 )$ 
\end{proof}

 \section{Decay of correlations }
 In this section and in the next we prove several statistical properties for our map: they are basically consequences of the distortion inequality got in the previous section matched with established techniques.

\begin{proposition}
The map $T$  enjoys polynomial decay of correlations (w.r.t.\ the Lebesgue measure $m$),   for H\"older continuous functions on $\mathbb{T}$. More precisely, for all H\"older $\varphi:\mathbb{T}\rightarrow \mathbb{R}$ and all $\psi\in L^{\infty}(\mathbb{T},m)$, we have:
$$
\left|\int (\varphi\circ T^n)\:\psi\:dm -\int \varphi\: dm\int \psi\:dm  \right|=\mathcal{O}\big(\frac{1}{n^{\frac{1}{\gamma-1}
}}\big)\, .
$$
\end{proposition}

\begin{proof}
We will use Lai-Sang Young's tower technique \cite{LSY}.
We build the  tower over the interval $I_0$ and we define the return time function  as the first return time:  $$\text{for all} ~x\in I_0,~~ 
R(x):=\min\{n\in \mathbb{N}^+~;~T^nx\in I_0\}:=\tau_{I_0}(x) \, .$$
The tower is thus defined by: 
$$\Delta=\{(x,l)\in I_0\times\mathbb{N}~;~l\leq\tau_{I_0}(x)-1 \} $$
and  the partition of the base $I_0$ is given by the cylinders $Z_{0,p}$ defined in the previous section. Recall that the dynamics on the tower is given by:
$$F(x,l)=
\left\{
\begin{aligned}
(x,l+1)~~~~~~&\text{if}~~~~l<\tau_{I_0}(x)-1\\
	(T^{\tau_{I_0}(x)}(x),0)~~~&\text{if}~~~~l=\tau_{I_0}(x)-1
\end{aligned}
\right.
$$
According to \cite{LSY}, the decay of correlations is governed by the asymptotics of $m\{x\in I_0~;~\tau_{I_0}(x)\geq n \}$ namely
$$m\{x\in I_0~;~\tau_{I_0}(x)>n\}=m(b_{-n},b_{n})\sim\frac{1}{\gamma}\left(\frac{2\gamma}{\gamma-1}\right)^{\frac{\gamma}{\gamma-1}}\frac{1}{(n-1)^{\frac{\gamma}{\gamma-1}}} \, .$$

Moreover  we must verify another important requirement of the theory; this will also be useful in the next section about limit theorems.  Let us first introduce the separation time $s(x,y)$ between two points $x$ and $y$ in $I_0$. Put $\hat T$ the first return map on $I_0$; we define $s(x,y)=\min_{n\ge 0}\{(\hat T^n(x), \hat T^n(y))  \ \mbox{lie in distinct} \ Z_{0,p}, p\ge 1\}$. We ask that $\exists C>0, \delta \in (0,1)$ such that $\forall x,y \in Z_{0,p}, p\ge 1$, we have
\begin{eqnarray}\label{s}
\left|\frac{D\hat T(x)}{D\hat T(y)}\right|\le \exp[C \delta^{s(\hat T(x), \hat T(y))}] \, .
\end{eqnarray}
Let us prove this inequality. Remember that the cylinder $Z_{0,p}$ is the disjoint union of two open components, $Z^+_{0,p}$ and  $Z^-_{0,p}$, which sit on the opposite sides of $0$. Suppose first that $x$ and $y$ stay in the same open component of some $Z_{0,p}$, $p\ge 1$,  and that 
 $s(\hat T(x),\hat T(y))=n$; then since the orbits (under $\hat T$) of the two points will be in the same cylinder  up to time $n-1$,  and on these cylinders $\hat T$ is monotone and uniformly expanding, $|D\hat T|\ge \beta>1$ (see footnote 1), we have  $|\hat T(x)-\hat T(y)|\le \beta^{-(n-1)}$.  Therefore by the distortion inequality  we get
\begin{equation}\label{separation}
\left|\frac{D \hat T(x)}{D \hat T(y)}\right|\le  \exp\left[ K \beta^{-(n-1)}\right]\le \exp[C \delta^{ s(\hat T(x),\hat T(y))}] \, ,
\end{equation}
where $C= K \beta $ and $\delta=\beta^{-1}$.
If instead $x,y$ lie in the two different open components of some   $Z_{0,p}$, $p\ge 1$, and again $s(\hat T(x),\hat T(y))=n$, this means that $-x$ and $y$ will have the same coding up to $n$; hence
$$
\left|\frac{D \hat T(x)}{D \hat T(y)}\right|=\left|\frac{D \hat T(-x)}{D \hat T(y)}\right| 
$$
$$
\le \exp[ K |\hat T(-x)-\hat T(y)| ]\le \exp\left[ K \beta^{-(n-1)} \right]\le \exp[C \delta^{ s(\hat T(x),\hat T(y))}]\, .
$$
According to \cite{LSY} the correlations decay satisfies 
$\left|\int (\varphi\circ T^n)\:\psi\:dm -\int \varphi\: dm\int \psi\:dm  \right|=\mathcal{O}(\sum_{k>n}m\{x\in I_0~;~\tau_{I_0}(x)\geq k\}$ and  the right hand side of this inequality behaves like
 $\mathcal{O}\big(n^{-\frac1{\gamma-1}}\big)$.
\end{proof}

{\em Optimal bounds} 
The previous result on the decay of correlations could be strengthened to produce a lower bound for the decay of correlations 
for integrable functions which vanish in a neighborhood of the indifferent fixed point. We will use for that the renewal technique introduced by Sarig \cite{SA} and succesively improved by Gou\"ezel \cite{GO2}. We first need that our original map is irreducible: this is a consequence of the already proved ergodicity, but one could shown directly by inspection that the countable Markov partition given by the preimages of zero has such a property. We moreover need additional properties that we directly formulate in our setting:
\begin{itemize}
\item Suppose we induce on $I_m=(a_{-m}, a_m)/\{0\}$ and call ${\cal Z}_m$  the Markov partition into the rectangles $Z_{m,p}$ with first return $p$. A cylinder $[d_0,d_1,\cdots, d_{n-1}]$ with $d_i\in {\cal Z}_m$ will be the set $\cap_{l=0}^{n-1}\hat T^{-i}d_l$. 

We first need that the jacobian of the first return map is locally H\"older continuous, namely that there exists $\theta<1$ such that:
$$
\sup |\log D\hat T(x)-\log D\hat T(y)|\le C \theta^n \, ,
$$
where the supremum is taken over all couples $x,y \in [d_0,d_1,\cdots, d_{n-1}]$, $d_i\in {\cal Z}_m$ and $C$ is a positive constant.
But this is  an immediate consequence of formula (\ref{separation}) with $\theta=\beta^{-1}$ and $C=K\beta$.  Using the separation time $s(\cdot,\cdot)$, we define $D_mf=\sup|f(x)-f(y)|/\theta^{s(x,y)}$, where $f$ is an integrable function on $I_m$ and the supremum is taken over all couples $x,y\in I_m$. We then put $||f||_{{\cal L}_{\theta,m}}\equiv ||f||_{\infty}+ D_mf$. We call ${\cal L}_{\theta,m}$ the space  of $\theta$-H\"older functions on $I_m$.
\item We need the so-called {\em big image} property, which means that the Lebesgue measure of the images, under $\hat T$, all the rectangles $Z_{m,p}\in {\cal Z}_m$ are uniformly bounded from below by a strictly positive constant. In our case, see section 3, these images are bounded from below by the length of the interval $(a_{-m}, a_m)$.
\item We finally need that $m(x\in I_m | \tau(x)>n)={\cal O}(n^{-\chi})$, for some $\chi>1$ (this is Gou\"ezel's assumption, which improves Sarig's one, asking for $\chi>2$). In our case by the construction developed in Sect. 3 we immediately get that
$m(x\in I_m | \tau(x)>n)=m( \cup_{p>n} Z_{m,p}) = (b_{-(m+n)}, b_{m+n}) \sim C (n+m)^{-b} = C n^{-b}(1+m/n)^{-b} \sim C n^{-b}$, where the constants $C$ and $b$ are the same as those given in the proof of Th. 4, precisely $C=\frac{1}{\gamma}\left(\frac{2\gamma}{\gamma-1}\right)^{\frac{\gamma}{\gamma-1}}$ and $b=\frac{\gamma}{\gamma-1}$.

\end{itemize} 
Under these assumptions, Sarig and Gou\"ezel proved a lower bound for the decay of correlations which we directly specialize to our map:
\begin{proposition}\label{OBU}
 There exists a constant $C$ such that for all $f$ which are $\theta$-H\"older and $g$ integrable and both supported in $I_m$ we have
 $$
 \left| Corr(\:f ,g\circ T^n) -(\sum_{k=n+1}^{\infty}m(x\in I_m | \tau(x)>n))\int g \:dm\int f\:dm  \right|\le C F_{\gamma}(n)||g||_{\infty}||f||_{{\cal L}_{\theta,m}}
 $$
 where $F_{\gamma}(n)=\frac{1}{n^{\frac{\gamma}{\gamma-1}}}$ if $\gamma<2$, $(\log n)/n^2$ if $\gamma=2$ and $\frac{1}{n^{\frac{2}{\gamma-1}}}$ if $\gamma>2$.\\Moreover, if $\int f\: dm=0$, then $\int (g\circ T^n)\:f\:dm={\cal O}(\frac{1}{n^{\frac{\gamma}{\gamma-1}}})$. Finally the central limit theorem holds for the observable $f$. 
 \end{proposition}
 \begin{remark}
 (i) Since when $m\rightarrow \infty$, $I_m$ covers mod-$0$ all the interval $(-1,1)$ we get an optimal decay of correlations of order ${\cal O}(\frac{1}{n^{\frac{1}{\gamma-1}}})$ for all integrable smooth enough functions which vanish in a neighborhood of $1$. \\
 (ii) The last sentence about the existence of the central limit theorem will be also obtained, using a different technique, in Proposition 5, part 2, (a).
 \end{remark}

\section{ Limit theorems }

Let us recall the notion of stable law (see \cite{F,Gou}): a stable law is the limit of a rescaled i.i.d process. More precisely, the distribution of a random variable $X$ is said to be stable if there exist an i.i.d stochastic process $(X_i)_{i\in\mathbb{N}}$ and some constants $A_n\in\mathbb{R}$ and $B_n>0$ such that in distribution:
$$
\frac{1}{B_n}\big(\sum_{i=0}^{n-1}X_i-A_n\big) \longrightarrow X \, .
$$
The kind of laws we are interested in can be characterized by their \textit{index} $p\in(0,1)\cup (1,2)$, defined as followed:
$$m(X>t)=(c_1+o(1))t^{-p} \, ,~~~~~~~~~~~~m(X<-t)=(c_2+o(1))t^{-p} \, ,$$
where $c_1\geq 0$ and $c_2\geq 0$ are two constants such that $c_1+c_2>0$, and by other two parameters:\\
\begin{equation*}
c=\left\{ 
\begin{aligned}
& (c_1+c_2)\Gamma(1-p) \cos(\frac{p\pi}2)   & \quad \quad  p&\in (0,1)\cup (1,2) \\
&   \frac12 & \quad \quad p&=2    
\end{aligned}
\right. , \quad \quad \quad \beta=\frac{c_1-c_2}{c_1+c_2}\, .
\end{equation*}
We will denote by $X(p,c,\beta)$ the law whose characteristic function is
$$
E(e^{X(p,c,\beta)})=e^{-c|t|^p\big(1-i\beta \text{sgn}(t)\tan(\frac{p\pi}{2})\big)}\, .
$$
\vspace{0.1cm}

\begin{proposition} Let us denote $S_n\varphi=\sum_{k=0}^{n-1}\varphi\circ T^k$, where $\varphi$ is an $\nu$-H\"older observable, with $\int \varphi(x)\:dx=0$.
\begin{enumerate}
\item If $\gamma < 2$ then the Central Limit Theorem holds for any $\nu>0$. That is to say there exists a constant $\sigma^2$ such that $\cfrac{S_n\varphi}{\sqrt{n}}$ tends in distribution to $\mathcal{N}(0,\sigma^2)$.

\item If $\gamma >2$ then:
\begin{enumerate}
\item If $\varphi(1)=0$ and $\nu>\frac{1}{2}(\gamma-2)$ then the Central Limit Theorem still holds.
 Moreover $\sigma^2=0$  iff there exists a measurable function $\psi$ such that $\phi=\psi\circ T-\psi$
\item If $\varphi(1)\neq 0 $ then $\cfrac{S_n\varphi}{ n^{\frac{\gamma-1}{\gamma}}  }$ converges in distribution to the stable law $X\big(p ,c, \beta\big)$ with:
 \begin{eqnarray*} 
  p&=& \frac{\gamma}{\gamma-1}\\
 c&=& \frac{1}{2\gamma}\left(\frac{2\gamma\varphi(1)}{\gamma-1}\right)^{\frac{\gamma}{\gamma-1}}\Gamma(\frac1{(1-\gamma)})\cos(\frac{\pi\gamma}{2 (\gamma-1)})
 \\
 \beta&=&\textrm{sgn}\,\varphi(1)
\end{eqnarray*}
\end{enumerate}

\item If $\gamma =2$ then:
\begin{enumerate}
\item If $\varphi(1)=0$  then the Central Limit Theorem  holds.
\item If $\varphi(1)\neq 0$ then there exist a constant $b$ such that  $\cfrac{S_n\varphi}{\sqrt{n\log{n}}}$ tends in distribution to $\mathcal{N}(0,b)$.
\end{enumerate}

\end{enumerate}
\end{proposition}

\begin{proof}
~
\begin{enumerate}

\item  
As a by-product of the tower's theory we get the existence of the central limit theorem whenever the rate of decay of correlations is summable (\cite{LSY}, Th.~4); this happens in our case for $\gamma <2$. As usual we should avoid that $\phi$ is a co-boundary. 
\item \begin{enumerate}
 
\item We proceed as in~\cite{Gou} Th.~1.3 where this result was proven
 for the Pomeau-Manneville parabolic maps of the interval. We defer the reader to Gou\"ezel's paper for the preparatory theory; we only prove here the necessary conditions for its application. We induce again on ${I_0}$ and we put $\varphi_{I_0}(x):=\sum_{i=0}^{\tau_{I_0}-1}\varphi(T^ix)$. We need:
\begin{enumerate}
\item $\phi$ must be locally $\theta$-H\"older on $ I_0$ (resp.~$\mathbb{T}$), with $\theta<1$, which means that there exists a constant $C$ such that $|\phi(x)-\phi(y)|\le C\theta^{s(x,y)}$ $\forall x,y \in I_0$ (resp.$\mathbb{T}$) with $s(x,y)\ge 1$. We extend the separation time $s(x,y)$ to the ambient space as follows: if $x,y\in \mathbb{T}$, call $\hat x, \hat y$ their first returns to $ I_0$. Whenever $T^ix, T^iy$ stay in the same element of the Markov partition $\{I_m\}_{m\in\mathbb{Z}}
$ until the first return to $I_0$, we put $s(x,y)=s(\hat x, \hat y)+1$; otherwise $s(x,y)=0$.
\item $m\{x\in I_0; \tau_{I_0}(x) >n\}={\cal{O}}(1/n^{\eta+1})$, for some $\eta>1$
\item $\varphi_{I_0}\in \mathcal{L}^2\big(I_0\big)\, .$
\end{enumerate}
Recall that the induced map $\hat T$ on $I_0$ is uniformly expanding with factor $\beta>1$; 
therefore for any couple of points $x,y\in \mathbb{T}$ we have $|x-y|_{\mathbb{T}}\le B \beta^{-s(x,y)}$, where $B$ is a suitable constant and $|\cdot|_{\mathbb{T}}$ denotes the distance on the circle. Using the H\"older assumption on $\phi$ we get $
|\phi(x)-\phi(y)|\le D |x-y|_{\mathbb{T}}^{\nu}\le E \beta^{-\nu s(x,y)}$, which shows that $\phi$ is locally H\"older with $\theta=\beta ^{-\nu}<1$.

The quantity in the second item above is exactly  $(b_n,b_{-n})$ for which we obtained in the previous section a bound of  order 
  $n^{-(\frac{\gamma}{\gamma-1})}$. Hence $\eta=\gamma/(\gamma-1)-1$.

To prove the third item denote $C_{\varphi}=\int_{I_0}|\varphi(x)|^2dx$  we obtain:
\begin{align*}
\int_{I_0^+}|\varphi_{I_0^+}(x)|^2\:dx &= C_{\varphi}+\sum_{p=2}^{+\infty}\int_{Z_{0,p}}\Big|\sum_{i=0}^{p-1}\varphi(T^ix)\Big|^2dx\\
&\lesssim C_{\varphi}+2\sum_{p=2}^{+\infty}\int_{b_p}^{b_{p-1}}\Big|\sum_{i=0}^{p-1}|T^ix- 1|_{\mathbb{T}}^{\nu}\Big|^2dx\\&\lesssim C_{\varphi}+2\sum_{p=2}^{+\infty}\int_{b_p}^{b_{p-1}}\Big|\sum_{i=0}^{p-1}|a_{i}-1|^{\nu}\Big|^2dx\\
&\lesssim C_{\varphi}+2\sum_{p=2}^{+\infty} m(b_p-b_{p-1})p^{2(-\frac{\nu}{\gamma-1}+1)}dx\\
&\lesssim C_{\varphi}+2\sum_{p=2}^{+\infty} p^{-(\frac{\gamma}{\gamma-1}+1)} p^{2(-\frac{\nu}{\gamma-1}+1)}dx \, .
\end{align*}
Finally if $\frac{2(-\nu+\gamma-1)}{\gamma-1}-\frac{\gamma}{\gamma-1}-1<-1$ (i.e.\ 
$\nu>\frac{1}{2}(\gamma-2)$) then $\varphi_{I_0}\in \mathcal{L}^2\big(I_0\big)$.
\item  Using the fact that  
$$m[u>n \varphi(-1)] =m(b_n,b_{-n})\sim\frac{1}{2\gamma}\left(\frac{2\gamma}{\gamma-1}\right)^{\frac{\gamma}{\gamma-1}}\frac{1}{n^{\frac{\gamma}{\gamma-1}}}$$
  and the proof in 2.(a), the result easily follows along the same lines of the proof of Th. 1.3 in \cite{Gou}. 
\end{enumerate}
\item This could also be argued as in the Proof of Th. 1.3 in \cite{Gou}.
\end{enumerate}
\end{proof}

\noindent
{\em Large deviations}.\\
The knowledge of the measure of the tail for the first returns on the tower (in our case built over $I_{0}$), will allows us to apply the results of Melbourne and Nicol~\cite{MN} to get the large deviations property for H\"older observables. 
Applied to our framework, their theorem states that if $m(x; \tau_{I_{0}}>n)={\cal O}(n^{-(\zeta+1)})$, with $\zeta>0$, then for all observables $\phi:[-1,1]\rightarrow \mathbb{R}$ which are H\"older and which we take of zero mean, we have the large deviations bounds:
\begin{proposition}
If $\gamma<2$ then the map $T$ verifies the following large deviations bounds:\\
(I) $\forall \epsilon>0$ and $\delta>0$, there exists a constant $C\ge 1$ (depending on $\phi$) such that
$$
m\left( \left|\frac{1}{n}\sum_{j=0}^{n-1}\phi(T^j(x))\right|>\epsilon\right)\le C n^{-(\zeta-\delta)}.
$$
(II) For an open and dense set of H\"older observables $\phi$, and for all $\epsilon$ sufficiently small, we have
$$
m\left( \left|\frac{1}{n}\sum_{j=0}^{n-1}\phi(T^j(x))\right|>\epsilon\right)\ge  n^{-(\zeta-\delta)}
$$
for infinitely many $n$ and every $\delta>0$.
\end{proposition}
\begin{remark}
The Melbourne and Nicol result has been recently strenghtened by Melbourne \cite{IM}; by adopting the same notation as above, he proved that 
whenever the observable $\phi$ is $L^{\infty}$ (with respect to the Lebesgue measure $m$), and $\zeta+1>0$, then for any $\epsilon$ there exists a constant $C_{\phi, \epsilon}$ such that
$$
m\left( \left|\frac{1}{n}\sum_{j=0}^{n-1}\phi(T^j(x))\right|>\epsilon\right)\le C_{\phi, \epsilon} n^{-\zeta}
$$
for all $n\ge 1$. Translated to our map, this means that {\em we have the large deviation property whenever $\gamma>1$}. Similar results have been obtained by Pollicot and Sharp \cite{PS} for the Pomeau-Manneville class of maps; hopefully they could be generalized in the presence of unbounded first derivaties.
\end{remark}
\section{ Recurrence}
{\em First returns}.\\
In the past ten years the statistics of first return and hitting times  have been widely used as  new and interesting tools to understand the recurrence behaviors in dynamical systems.  Surveys of the latest results and some 
historical background can be found in~\cite{HSV, HV, A}.

Take a ball $B_r(x)$ or radius $r$ around the point $x\in \mathbb{T}$ and consider the first return $\tau_{B_r(x)}(y)$ of the point $y\in B_r(x)$ into the ball. If we denote with $m_r$ the conditional measure to $B_r(x)$, we ask whether there exists the limit of the following distribution when $r\rightarrow 0$\footnote{We call it distribution with abuse of language; in probabilistic terminology we should rather take $1$  minus that quantity.}:
$$
F^e_r(t)= m_r\left(y\in B_r(x); \tau_{B_r(x)} m(B_r(x))>t\right).
$$

The distribution $F^h_r(t)$ for the first {\em hitting} time (into $B_r(x)$) is defined analogously just taking $y$   and the probability $m$ on the whole space $\mathbb{T}$. 

A powerful tool to investigate such distributions for non-uniformly expanding and hyperbolic systems is given by the conjunction of the following results, which   reduce the computations to induced subsets.
\begin{itemize}
\item Suppose $(T,X,\mu)$ is an ergodic measure preserving transformation of a smooth Riemannian manifold $X$; take $\hat X\subset X$ an open set and equip it with the first return map $\hat T$ and with the induced (ergodic) measure $\hat \mu$. For $x\in \hat X$ we consider the ball $B_r(x)$ ($B_r(x)\subset\hat X$)
 around it and we write $\hat \tau_{B_r(x)}(y)$ for the first return of the point $y\in B_r(x)$ under $\hat T$. We now consider the distribution of the first return time  for the two variables $ \tau_{B_r(x)}$ and $\hat \tau_{B_r(x)}$ in the respective probability spaces ($B_r(x), \mu_r$) and ($B_r(x), \hat \mu_r$) (where again the subindex $r$ means conditioning to the ball $B_r(x)$), as :  $F^e_r(t)= \mu_r(y\in B_r(x)); \tau_{B_r(x)}(y) \mu(B_r(x))>t)$ and $\hat F^e_r(t)= \hat \mu_r(y\in B_r(x));  \hat \tau_{B_r(x)}(y)\hat \mu(B_r(x))>t)$.\\

In \cite{BSTV} it is proved the following result:
suppose that for $\mu$-a.e. $x\in \hat X$ the distribution $\hat F^e_r(t)$  converges pointwise  to the continuous functions  $f^e(t)$  when $r\rightarrow 0$ (remember that  the previous distribution depend on $x$ {\em via} the location of the ball $B_r(x)$); then we have as well $F^e_r(t)\rightarrow f^e(t)$ and the convergence is uniform in $t$\footnote{The result proved in~\cite{BSTV} is slightly more general since it doesn't require the continuity of the asymptotic distributions over all $t\ge 0$. We should note instead that we could relax the assumption that $\hat X$ is open just removing from it a set of measure zero, which will happen on our induced sets $I_m$.}.
We should note that whenever we have the distribution $f^e(t)$ for the first return time we can insure the existence of the weak-limit distribution for the first hitting time $F^h_r(t)\rightarrow f^h(t)$ where
$f^h(t)=\int_0^t(1-f^e(s))ds,\; t\ge 0$~\cite{HLV}.

{\sc Note}: From now on we will say that we have $f^{e,k}(t)$ as  {\em limit distributions for balls}, if we get them in the limit $r\rightarrow 0$ and for $\mu$-almost all the centers $x$ of the balls $B_r(x)$.
\item The previous result is useful if we are able to handle with recurrence on induced subsets, see \cite{BV, BT} for a few applications. Induction for one-dimensional maps often produces piecewise monotonic maps with countably many pieces. An interesting class of such maps are the Rychlik's maps~\cite{Ri} : in  \cite{BSTV} Def.~3.1
the underlying measure is conformal. When the conformal measure is the Lebesgue measure $m$, then  Rychlik's maps could be characterized in the following way:

 Let $T:Y\rightarrow X$ be a continuous map, $Y\subset X$ open and dense, $m(Y)=1$ and $X$ is the 
 unit interval or the circle. Suppose there exists a countable family of pairwise disjoint open intervals 
 $Z_i$ such that $Y=\bigcup_{i\le 1}Z_i$ and $T$ is: (i) $C^2$ on each $Z_i$; (ii) uniformly expanding:
  $\inf_{Z_i}\inf_{x\in Z_i}|DT(x)|\ge \beta>1$; (iii) Var$(g)<\infty$, where $g=1/|DT(x)|$ 
  when $x\in Y$ and $0$ otherwise (Var~$g$ denotes the total variation of the function $g: \mathbb{R} \rightarrow\infty$).
  
  In~\cite{BSTV} Th.~3.2 it was shown that such maps have exponential return time statistics around balls 
  (i.e.~$f^e(t)=f^k(t)=e^{-t}$), whenever the invariant measure is absolutely continuous w.r.t. $m$  and moreover this invariant measure  is mixing.
\end{itemize}
  Before we formulate our next result for the maps $T$ investigated in this paper let us prove the 
  following lemma.

  \begin{lemma}
  The map $\hat T$ is Rychlik on the cylinders $I_m$, $m\in\mathbb{Z}$ and the
  variation of $|D\hat T|$ is finite on each of them.
  \end{lemma}
    \begin{proof} (see~\cite{BSTV}). Let us consider the cylinder $I_m$ and partition it 
    into the cylinders $Z_{m,p}$ with first return $p\ge 1$, as we did in the second section; then we have for 
    the variation on $I_m$
  $$
  \mbox{Var} \frac{1}{|D\hat T|}\le \sum_{Z_{m,p}}\int_{Z_{m,p}}\frac{| D^{2}\hat T(t)|}{|D\hat T(t)|^2}dt+2\sum_{Z_{m,p}} \sup_{Z_{m,p}}\frac{1}{|D\hat T|}\, .
  $$
  By the distortion bound proved in the second section we have that 
  $$
  e^{2K}\ge \left|\frac{D\hat T(x)}{D\hat T(y)}\right|\ge \left|\int_x^y \frac{D^{2}\hat T(t)}{D\hat T(t)}dt\right|\ge \int_x^y \frac{|D^{2}\hat T(t)|}{D\hat T(t)}\,dt
  $$
   for any $x,y\in Z_{m,p}$, since the first derivative is always positive and the second derivative has the same sign for all the points in the same cylinder. But this immediately implies that $\int_{Z_{m,p}}\frac{| D^{2}\hat T(t)|}{|D\hat T(t)|^2}\,dt\le\sup_{Z_{m,p}}\frac{1}{|D\hat T|}e^{2K}$. Using Remark(\ref{remark}) we can restrict to $Z_m^-$. Since $\hat T$ maps  $Z^-_{m,p>1}$ diffeomorphically onto $(a_{m-1},a_{m})$ and $Z^-_{m,1}$ onto $(a_{-(m-1)},a_m)\supset (a_{m-1},a_{m}) $ there will be a point $\xi$ for which $D\hat T(\xi) m( Z_{m,p} ) \ge m(a_{m-1},a_{m})$. Applying the bounded distortion estimate one more time, we get $\sup_{Z_{m,p}}\frac{1}{|D\hat T|}\le \frac{e^{2K} m(Z_{m,p}) }{ m(a_{m-1},a_{m})}$. We finally obtain
  $$
  \mbox{Var} \frac{1}{|D\hat T|}\le \frac{e^{2K}(2+e^{2K})}{m(a_{m-1},a_{m})}\sum_{Z_{m,p}} m( Z_{m,p}) < \infty \, .
  $$
  \end{proof}
  
  \noindent The following result now follows by~\cite{BSTV} Theorem~3.2.

  \begin{proposition}
  The map $T$ has exponential return and hitting time distributions with respect to the measure $m$ 
  provided  $\gamma>1$. 
  \end{proposition}
 
   \noindent
{\em Number of visits.}\\
Let us come back to the general framework introduced in Sect. 5.1 with the two probability spaces $(X,T,\mu)$ and $(\hat X, \hat T, \hat \mu)$. We now introduce the random variables $\xi^e_r$ and $\hat\xi^e_r$ which count the number of visits of the orbits of a point $y\in B_r(x)$ to the ball itself and up to a certain rescaled time. Namely:
$$
\xi^e_r(x,t) \equiv \sum_{j=1}^{\left[ t\over\mu(B_r(x))\right]}\chi_{B_r(x)} \left( T^j(y) \right)\, ,
$$
where $\chi$ stands for the characteristic function and $x\in X$. If we take $x\in \hat X$ we can define in the same manner the variable $\hat \xi^e_r(x,t)$ by replacing the action of $T$ with that of $\hat T$. We now introduce the two distributions
$$
G^e_r(t,k) = \mu_r(x; \xi^e_r(x,t)=k), \ \hat G^e_r(t,k) = \hat \mu_r(x; \hat \xi^e_r(x,t)=k)\, ,
$$
where again the index $r$ for the measures means conditioning on $B_r(x)$. It is proved in \cite{BSTV} that whenever the distribution $\hat G^e_r(t,k)$ converges weakly (in $t$) to the function $g(t,k)$ and for almost all $x\in \hat X$, the same happens, with the same limit,  to the distribution $G^e_r(t,k)$. For systems with strong mixing properties the limit distribution is usually expected to be Poissonian~\cite{HSV, HV,HV2, A}: $\frac{t^k e^{-t}}{k!}$. 

In~\cite{FFT} it was shown that Rychlik maps enjoy Poisson statistics for the limit distribution of the variables $\xi^e_r$  and whenever the center of the ball is taken a.e.. Hence we get the following result.
\begin{proposition}
Let $\gamma>1$. Then for $m$-almost every $x$ the number of visits to the balls $B_r(x)$ converges to the Poissonian distribution as  $r\rightarrow 0$.
\end{proposition}
\vspace{1cm}
\noindent
{\em Extreme Values}.\\
The last quoted paper \cite{FFT} contains another interesting application of the statistics of the first hitting time that we could apply to our map $T$ too. Let us first briefly recall the Extreme Value Theory. Given the probability measure preserving dynamical system $(X,T,\mu)$ and the observable $\phi: X\rightarrow \mathbb{R}\cap\{\pm\infty\}$, we consider the process $Y_n=\phi\circ T^n$ for $n\in \mathbb{N}$. Then we define the partial maximum $M_n\equiv \max\{Y_0,\cdots,Y_{n-1}\}$  and we look if there are normalising sequences $\{a_n\}_{n\in \mathbb{N}}\subset \mathbb{R}^+$ and $\{b_n\}_{n\in \mathbb{N}}\subset \mathbb{R}$ such that
$$
\mu(\{x: a_n(M_n-b_n)\le y\})\rightarrow H(y)
$$
for some non-degenerate distribution function $H$: in this case we will say that an {\em Extreme Value Law} (EVL) holds for $M_n$. If the  variables $Y_n$ were i.i.d., the classical extreme value theory prescribes the existence of only three types of non-degenerate asymptotic distributions for the maximum $M_n$ and under linear normalisation, namely:
\begin{itemize}
\item Type 1: $EV_1= e^{-e^{-y}}$ for $y\in \mathbb{R}$, which is called the {\em Gumbel} law.
\item Type 2: $EV_2= e^{-y^{-\alpha}}$ for $y>0$, $EV_2=0$, otherwise, where $\alpha>0$ is a parameter, which is called {\em Frechet} law.
\item Type 3: $EV_3= e^{-(-y)^{\alpha}}$ for $y\le 0$, $EV_3=1$, otherwise, where $\alpha>0$ is a parameter, which is called {\em Weibull} law.
\end{itemize}
From now on we will take $X$ as a Riemannian manifold with distance $d$ and $\mu$  an absolutely continuous (w.r.t. Lebesgue) probability invariant measure. Moreover consider the observable $\phi$ of the form $\phi(x)= g(d(x,\xi))$,  where $\xi$ is a chosen point in $X$. The function $g:[0, \infty)\rightarrow \mathbb{R}\cup \{+\infty\}$ is a strictly decreasing bijection in a neighborhood of $0$ and it has $0$ as a global maximum (eventually $+\infty$). The function $g$ could be taken in three classes; we defer to \cite{FFT} for the precise characterization. Important representatives of such classes (denoted by the indices {1,2,3}) are $g_1(x)= -\log(x)$; $g_2(x)=x^{-1/\alpha}$ for some $\alpha>0$; $g_3(x)=D-x^{-1/\alpha}$, for some $D\in \mathbb{R}$ and $\alpha>0$. We also remind the distribution of the first hitting time $F^h_r(t)$ into the ball $B_r(x)$ introduced above; we say that a system enjoys exponential hitting time statistics (EHTS) if $F^h_r(t)$ converges point wise to $e^{-t}$ for $\mu$-a.e. $x\in X$ (we saw before that it is equivalent to get the exponential limit distribution for the first return time). We are now ready to state the result in \cite{FFT} which establishes an equivalence between the EHTS and the EVL; we will be in particular concerned with the following implication: {\em suppose  the system $(X,T,\mu$) has EHTS; then it satisfies an EVL  for the partial maximum $M_n$ constructed on the process $\phi(x)= g(d(x,\xi))$, where $g$ is taken in  one of the three classes introduced above. In particular if $g=g_i$ we have an EVL for $M_n$ of type $EV_i$}.

 Of course this result can be immediately applied to the mapping $T$ under investigation in this paper.

\section{Generalizations}
 As mentioned in the Introduction the original paper by Grossmann and Horner~\cite{GH} dealt with 
different Lorenz-like maps $S$ which map $[-1,1]$ onto itself with two surjective symmetric branches 
 defined on the half intervals $[-1, 0]$ and $[0, 1]$. 
 They have the following local behaviour:
 \noindent
\begin{eqnarray*}
S(x)&\sim &1-b|x|^{\kappa}, \ x\approx 0, \ b>0\\
 S(x)&\sim &-x+a|x-1|^{\gamma}, \ x\approx 1_-, \ a>0\\
 S(x)&\sim &x+a|x+1|^{\gamma}, \ x\approx -1_+
\end{eqnarray*}
 where $\kappa\in(0,1)$ and $\gamma>1$ are two parameters.
We also require that\\ (i) in all points $x\neq -1, 1$ the absolute value of the derivative is strictly bigger than $1$.\\ (ii) $S$ is strictly increasing on $[-1,0]$, strictly decreasing on $[0,1]$ and convex on the two intervals $(-1,0), (0,1)$\\

 \noindent The map has a cusp at the origin where the 
 left and right first derivatives diverge to $\pm\infty$ and  the fixed point $-1$ is parabolic  (Fig.~2).
Although the map $S$ is Markov with respect to the partition $\{[-1, 0],[0, 1]\}$ it will be more convenient to
 use a countable Markov partition whose endpoints are given by suitable preimages of $0$ (see below). 
 
   The reflexion symmetry of the map $T$ in Sect. 2 was related to the invariance of the Lebesgue measure. We  do not really need that the map $S$ is symmetric with respect to the origin. We did this choice to get only two scaling exponents ($\kappa$ and $\gamma$) in $0$ and in $\pm1$. This implies in particular the same scalings for the preimages of $0$ on $(-1,0)$ and $(0,1)$. If the left and rigt branches are not anymore symmetric, still preserving the Markov structure and the presence of indifferent points and of a point with unbounded derivative,  one should play with at most four scaling exponents giving the local behavior of $S$ in $0$ and  $\pm 1$. 
 \begin{figure}[t]
\centerline{\epsfxsize=12.cm \epsfbox{lorenz-bn.eps}}
\end{figure}
\noindent 
\\    
\newline
We denote by $S_1$ (resp.~$S_2$) the restriction of $S$ to $[-1,0]$ (resp.~$[0,1]$) and define $a_{0+}=S^{-1}_20; \  a_{0-}=S^{-1}_10;  \ a_{-p}=S^{-p}_1a_{0-};  \ a_p=S^{-1}_2S^{-(p-1)}_1a_{0-}$
for $p=1,2,\dots$. It follows that $Sa_{-p}=Sa_p=a_{-(p-1)}$. In the same way as we did in the first section we define the sequence $b_p, p\ge 1$ as: $S b_{\pm p}=a_{p-1}$.  The countable Markov partition, $\mod m$, will be
 $\left\{ (a_{-p}, a_{-(p-1)}): p\ge1\right\}\cup\left\{(a_p, a_{p+1}):p\ge1\right\}\cup\{I_{0}\},I_0\equiv (a_{0-},a_{0+})/\{0\} $. 

\noindent From the local behaviors one gets the following scaling relations
\begin{eqnarray*}
 a_p=-a_{-p}&\sim&  1-\left( \frac1{a(\gamma-1)} \right)^{\frac1{\gamma-1}} \frac1{p^{\frac1{\gamma-1}}}  \\
a_{p}-a_{p+1}&\sim & a\left( \frac1{a(\gamma-1)} \right)^{\frac{\gamma}{\gamma-1}} \frac1{p^{\frac{\gamma}{\gamma-1}}}  \\
\\
\\
 b_p =-b_{-p}&\sim& \left( \frac1{ab^{(\gamma-1)}(\gamma-1)} \right)^{\frac1{k(\gamma-1)}} \frac1{p^{\frac1{k(\gamma-1)}}}   \\
b_{p}-b_{p-1}&\sim& \frac1{k \left( ab^{(\gamma-1)} \right)^{\frac1{k(\gamma-1)}}} \left(\frac1{\gamma-1} \right)^{\frac{k(\gamma-1)+1}{k(\gamma-1)}} \frac1{p^{\frac{k(\gamma-1)+1}{k(\gamma-1)}}}  
 \end{eqnarray*}

 \noindent
{\em Bounded distortion}.
 The distortion is estimated in  the same way as it was done in the proof of Proposition 1, with however two differences: 
 \begin{itemize}
 \item The role of Remark1 is played here by the monotonicity of the right branch: whenever $x,y$ sit on different components we can just note that $|DS(-x)|=|DS(x)|$  and that after one iteration $S(x)=S(-x)$. \footnote{In the asymmetric case $|DS(-x)| \ne |DS(x)|$ but still after one iterate $S(x)$ and $S(y)$ sit on the same side. This imply that  multiplying by the appropriate factor we can treat the asymmetric case in the same way as the symmetric one.}  
\item
Let us consider again the step from the first to the second upper bound in (\ref{B1}): we simply discarded the denominator given by the infimum of the first derivative over the sets with given first return time, since it was ininfluent for the map $T$. Instead it will now plays an important role since it makes bounded the following ratio since, as it is easy to check, :

$$
\frac{|b_{n+1}-b_n|\sup_{(b_{n+1},b_n)}|D^2S|}{\inf_{(b_{n+1},b_n)}|DS|}= {\cal O}(\frac1n)\, .
$$

\end{itemize}
\noindent
{\em Invariant measure and decay of correlations}. \  An important difference with the map on the circle is that we are not guaranteed that the Lebesgue measure $m$ is anymore  invariant; so we have to build an absolutely continuous invariant measure $\mu$. Fortunately the tower's techniques helps us again. If the tail of the return time on the base of the tower is $m$-summable and the distortion is bounded, it follows the existence of such $\mu$. To be more precise let us induce on the cylinder $I_{0}$.  A  subcylinder $Z_p$ of $I_{0}$ with first return time $p$ will have the form\footnote{We would like to note that, contrarily to the map $T$ investigated in the previous sections,  the first return map $\hat S$ for $S$ on $I_0$ is not onto $I_0$ on each cylinder $Z_p$ with prescribed first return time. In fact $\hat S$ maps all the cylinders $(b_{p-1}, b_p)$ and ($b_{-p}, b_{-(p-1)})$ onto $(a_{0-}, 0)$, but  it maps the cylinders $(a_{0-}, b_{-1})$ and $(b_1, a_{0+})$ onto $(0, a_{0-})$. Nevertheless $\hat S$ is an irreducible Markov map, as it is easy to check. If one wants a genuine first return Bernoulli map, one should induce over $(a_{0-}, 0)$: the cylinders with given first return time are simply slightly more complicated to manage with.}
\begin{eqnarray}\label{SI}
Z_1&=& (a_{0-},b_{-1})\cup(b_1,a_{0+})\\
Z_p&=&(b_{-(p-1)},b_{-p})\cup(b_{p},b_{p-1}) ~~~p>1\, .\nonumber
\end{eqnarray}

Consequently the Lebesgue measure of the points in $I_{0}$ with first return bigger than $n$ scales like
$$
m(x\in I_{0}; \tau_{I_{0}}(x)> n)\approx \frac{1}{n^{\frac{1}{\kappa(\gamma-1)}}}
$$
We can thus invoke Th.~1 in Lai-Sang Young's paper~\cite{LSY} to get :
\begin{proposition}
Let us consider the map $S$ depending upon the parameters $\gamma$ and $\kappa$. Then
for $0< \kappa< \frac{1}{\gamma-1}$  (or for $0< \kappa< 1$, when $\gamma \le 2$), we get the existence of an absolutely continuous invariant measure $\mu$ which mixes polynomially fast on H\"older observables with rate $\mathcal{O}\big(n^{-\frac{1-\kappa(\gamma-1)}{\kappa (\gamma-1)}}\big)$.\\
The map has exponential return and hitting times distributions and Poissonian statistic for the limit distribution of the number of visits in balls.
\end {proposition}

{\em Optimal bounds.} 
As we did in the previous section the  result on the decay of correlations could be strengthened to produce a lower bound  for the decay of correlations 
 for integrable functions which vanish in a neighborhood of the indifferent point using  the renewal technique introduced in \cite{SA} and \cite{GO2}. The only difference with the  previous section is that now Lebesgue measure is not  invariant and thus we additionally need to show that the invariant density $\rho$ is Lipschitz in the region of inducing $I_m=(a_{-m},a_m)$. This is proved by first noting that the induced density $\widehat{\rho}$ is Lipschitz (see Eq.(\ref{mandrake}) below ) and then using the fact   that $\rho(x)=C_r \widehat{\rho}(x)$ for $x\in I_m$, that this is a direct consequence of Eq.(\ref{reldens}) below. 

Under these assumptions we get the analogous of Proposition \ref{OBU} above:
\begin{proposition}
 There exists a constant $C$ such that for all $f$ which are $\theta$-H\"older and $g$ integrable and both supported in $I_m$ we have
 $$
 \left| Corr(\:f ,g\circ T^n) -(\sum_{k=n+1}^{\infty}m(x\in I_m | \tau(x)>n))\int g\: dm\int f\:dm  \right|\le C F_{\gamma}(n)||g||_{\infty}||f||_{{\cal L}_{\theta,m}}
 $$
 where \\
 $F_{\gamma}(n)=
 \left\{ 
 \begin{aligned}
& n^{-\frac{1}{k(\gamma-1)}} & \hspace{0.3cm} \text{if}~~& 0 <\kappa<\frac{1}{2(\gamma-1)}\quad  (\text{or}\quad  0<\kappa<1, \quad\text{when} \quad  \gamma\le 3/2) \\
&(\log n)/n^2 & \hspace{0.6cm}  \text{if} ~~&\gamma=\frac1{2(\gamma-1)}\\
&n^{-\frac{2}{k(\gamma-1)}+2}& \hspace{0.3cm}  \text{if} ~~ &\frac1{2(\gamma-1)}<\kappa<\frac1{\gamma-1}\quad (\text{or}\quad \frac1{2(\gamma-1)}<\kappa<1, \text{when} \quad 3/2<\gamma\le 2)
\end{aligned}
\right. 
$\\ Moreover, if $\int f \:dm=0$, then $\int (g\circ T^n)\:f\:dm={\cal O}(\frac{1}{n^{\frac{1}{k(\gamma-1)}}})$. Finally the central limit theorem holds for the observable $f$. 
 \end{proposition}

 \begin{remark}
 (i) Since when $m\rightarrow \infty$, $I_m$ covers mod-$0$ all the interval $(-1,1)$ we get an optimal decay of correlations of order ${\cal O}(n^{-\frac{1-k(\gamma-1)}{k(\gamma-1)}})$ for all integrable smooth enough functions which vanish in a neighborhood of $-1$ and of $1$. \\
 (ii) The last sentence about the existence of the central limit theorem will be also obtained  in Proposition \ref{LimitLorenz}, part 2, (a).
 \end{remark}

\noindent {\em Limit theorems}\ Following the corresponding arguments in  section 3 we have
\begin{proposition}\label{LimitLorenz} Let us denote $S_n\varphi=\sum_{k=0}^{n-1}\varphi\circ T^k$, where $\varphi$ is an $\nu$-H\"older observable, with $\int \varphi(x)\:dx=0$.
\begin{enumerate}
\item If $0<\kappa<\frac{1}{2(\gamma-1)}$ (or   $0<\kappa<1$, when $\gamma\le 3/2$), then the Central Limit Theorem holds for any $\nu>0$, nameky there exists a constant $\sigma^2$ such that $\cfrac{S_n\varphi}{\sqrt{n}}$ tends in distribution to $\mathcal{N}(0,\sigma^2)$.
\item If $\frac1{2(\gamma-1)}<\kappa<\frac1{\gamma-1}$ (or $\frac1{2(\gamma-1)}<\kappa<1$, when $3/2<\gamma\le 2$), then:
\begin{enumerate}
\item If $\varphi(-1)=0$ and $\nu>\frac{1}{2\kappa(\gamma-1)} $ then the Central Limit Theorem still holds.
 Moreover $\sigma^2=0$  iff there exists a measurable function $\psi$ such that $\phi=\psi\circ T-\psi$
\item If $\varphi(-1)\neq 0 $ then $\cfrac{S_n\varphi}{ n^{\frac1p}  }$ converges in distribution to the stable law $X\big(p ,c, \beta\big)$ with:
 \begin{eqnarray*} 
  p&=& \frac{1}{\kappa(\gamma-1)} \\
 c&=&  \rho(0)\left( \frac{\varphi(-1)}{ab^{(\gamma-1)}(\gamma-1)} \right)^{\frac1{k(\gamma-1)}}  
 \Gamma(1-p)\cos(\frac{\pi p}{2}) \\
 \beta&=&\textrm{sgn}\varphi(-1)
\end{eqnarray*}
where the density in $0$, $\rho(0)$, is always of order $1$ (see next section).
\end{enumerate}
\item If $k=\frac1{2(\gamma-1)}$ then:
\begin{enumerate}
\item If $\varphi(-1)=0$  then the Central Limit Theorem  holds.
\item If $\varphi(-1)\neq 0$ then there exist a constant $b$ such that  $\cfrac{S_n\varphi}{\sqrt{n\log{n}}}$ tends in distribution to $\mathcal{N}(0,b)$.
\end{enumerate}
\end{enumerate}
\end{proposition}

\noindent {\em Large deviations.}  Large deviations results can be derived following the corresponding arguments in previous sections. In particular, and by using the recent result by Melbourne \cite{IM}, we can state that for (Lebesgue) $L^{\infty}$ observables, the large deviation property holds with polynomial decay at a rate which is given by that of the decay of correlations; for our Lorenz maps it is of order 
$n^{-\frac{1-\kappa(\gamma-1)}{\kappa (\gamma-1)}}$, provided that $0< \kappa< \frac{1}{\gamma-1}$  (or  $0< \kappa< 1$, when $\gamma \le 2$).
\newline

\noindent {\em Densities.}  A heuristic analysis of the density $\rho$ of the measure $\mu$ was done in~\cite{GH}. 
According to Th.~1 in~\cite{LSY} the induced map $\hat S$ has a  density   $\hat \rho$ bounded away from $0$ and $\infty$ which additionally  verifies, for any two points $x,y$ in a cylinder with given first return time:
\begin{equation}\label{mandrake}
\left|\frac{\hat \rho(x)}{\hat \rho(y)}-1\right|\le C \beta^{s(x,y)}
\end{equation}
where $C>0$ depends on the map and $\beta<1$ and $s(\cdot,\cdot)$ are as in Sect.~3 (separation times). 
Note  that we could get the same result by observing that our induced maps are Rychlik 
(which was proved in Sect.~5), and for such maps Kowalski~\cite{KO} showed that the density 
is of bounded variation and bounded away from zero on the support of the invariant measure. 
What is instead the behavior of $\rho$. Is $\rho$ bounded from below away from $0$ too? Since we are working with the induced map, it is well known how to reconstruct the invariant measure $\mu$ if we are able to control the subset on the induced space with given first return. By applying this formula to our induced space $I_{0}$ we get:
\begin{equation}\label{reldens}
\mu(B)=C_r \sum_i \sum_{j=0}^{\tau_i-1} \hat \mu (S^{-j}(B)\cap Z_i)
\end{equation}
where $B$ is any Borel set in $[-1,1]$, $\hat \mu$ is the $\hat S$-invariant absolutely continuous measure on $I_{0}$ and the first sum runs over the cylinders $Z_i$ with prescribed first return time $\tau_i$ and whose union gives $I_{0}$. The normalising constant $C_r=\mu(I_0)$ satisfies $1=C_r \sum_i \tau_i \hat \mu(Z_i)$.
Since, as we said above, $\hat \mu$ is uniformly equivalent to $m$ on $I_{0}$, we will use the latter measure in the next computations. Notice that the  terms in the sum defining $C_r$ scale as 
$ \mathcal{O}(n^{-\frac1{\kappa(\gamma-1)}})$.

To obtain the asymptotics of the density in the vicinity of the (interesting) points $\pm 1$ and $0$ we
proceed as follows. 
We first note that in order to estimate the $\mu$-measure of of a set $B$ we need to consider only the cylinders $Z_p$ of $I_{0}$  which iterates  will have non-empty intersection  with B before they return to $I_0$. This immediately implies that
$\mu(B)=C_r\widehat{\mu}(B)$ if $B\subset I_0$. It follows that
$$
\mu\left( (b_{\pm(n+1)},b_{\pm n}) \right) \approx C_r m\left( (b_{\pm(n+1)},b_{\pm n}) \right) 
$$
 and thus the density $\rho(x)$, $x$ close to $0$,  is of order $1$. 
In  a similar way  we estimate the $\mu$-measure of the cylinder $(a_{n-1}, a_n)$ (for big $n$)
near the point $1$; 
we get that
  $S^{-(1)}(a_{n-1},a_n)\cap Z_{n+1}=Z_{n+1}$  
 is   the only
  possible non-empty intersection of the preimage 
 $S^{-j}(a_{n-1}, a_n)$ with $Z_p$,  for every $p$ 
 and for $0\le j\le p-1$.
Therefore we get:
$$
\mu((a_{n-1}, a_n))\approx C_r m(Z_{n+1})
\approx n^{-\frac{1-\kappa+\kappa\gamma}{\kappa(\gamma-1)}}
$$
The density on $(a_{n-1}, a_n)$ is given by $\rho((a_{n-1}, a_n))\approx \frac{\mu\left( (a_{n-1},a_n) \right)}{m\left( (a_{n-1},a_n) \right)}\approx n^{-\frac{1-\kappa}{\kappa(\gamma-1)}}$.\\
We now study the density in the neighborhood of $-1$, by considering the cylinder $(a_{-n}, a_{-n+1})$, for large $n>0$.
The  cylinders $Z_p$ of $I_{0}$  whose iterates  will have non-empty intersection  with $(a_{-n},a_{-n+1})$ before they return to $I_0$,  have $ p \ge n+2$.
Therefore we get in the usual way:
$$
\mu((a_{-n}, a_{-n+1}))\approx C_r\sum_{p=n+2}^{\infty} m(Z_{p}))
\approx n^{-\frac{1}{\kappa(\gamma-1)}}
$$
The density in $(a_{-n}, a_{-n+1})$ is given by $\rho((a_{-n}, a_{-n+1}))
\approx \frac{\mu\left( (a_{-n}, a_{-n+1}) \right)}{m\left( (a_{-n}, a_{-n+1}) \right)}
\approx n^{-\frac{1-\kappa\gamma}{\kappa(\gamma-1)}}$.
\newline
\newline
 Let us summarize these facts.
\begin{proposition}
Let us consider the map $S$ with $\gamma>1$ and $0< \kappa\ < \frac{1}{\gamma-1}$ (or $0< \kappa\ < 1$ when $\gamma\le 2$). We have
\begin{itemize}
\item When $x\rightarrow 1$ the density $\rho\equiv \rho(x)\rightarrow 0$ 
\item When $x\rightarrow -1$ the density  verifies:\\
(i) if $\kappa=\frac{1}{\gamma}$ then $\rho=\mathcal{O}(1)$\\
(ii) if $\frac{1}{\gamma}<\kappa$, then $\rho \rightarrow \infty$\\
(iii) if $\frac{1}{\gamma}>\kappa$, then $\rho \rightarrow 0$
\item The density is always of order $1$ in the neighborhood of $0$.
\end{itemize}
\end{proposition}
Note that our Proposition fits with the density found by Hemmer for the map (\ref{HH}); for this map and its circle companion (1) the correlations decay as $n^{-1}$.

\vspace{2cm}

{\em Acknowledgment}
We warmly thank R. Artuso  who introduced us to the maps studied in this paper and with whom we had  interesting discussions. S.V. and G.C. acknowledge a financial support of the GDRE (CNRS) "Grefi-Mefi". S.V. also thanks M. Gianfelice for useful discussions and for having showed him the reference \cite{PM}.

\vspace{0.3cm}
$$
-----------------------------------------------
$$
\begin{itemize}
\item Giampaolo \textsc{Cristadoro}: Mathematics Department, Universit\`a di Bologna, Bologna, Italy.\\
\url{cristadoro@dm.unibo.it} 
\item Nicolai \textsc{Haydn}: Mathematics Department, USC, Los Angeles, 90089-1113, USA.\\
 \url{nhaydn@math.usc.edu}
\item Philippe \textsc{Marie}: UMR-6207 Centre de Physique Th\'eorique, CNRS, Universit\'es d'Aix-Marseille I, II, Universit\'e du Sud, Toulon-Var, France.\\
\url{pmarie@cpt.univ-mrs.fr}
\item  Sandro \textsc{Vaienti}: UMR-6207 Centre de Physique Th\'eorique, CNRS,
Universit\'es d'Aix-Marseille I, II, Universit\'e du Sud, Toulon-Var and FRUMAM, F\'ed\'ederation de
Recherche des Unit\'es de Math\'ematiques de Marseille; address: CPT, Luminy Case 907, F-13288
Marseille Cedex 9, France.\\
 \url{ vaienti@cpt.univ-mrs.fr}
\end{itemize}

\end{document}